\documentclass{amsart}
\pdfoutput=1
\usepackage[utf8]{inputenc}
\usepackage[english]{babel}
\usepackage[usenames,dvipsnames]{xcolor}
\usepackage{graphicx,subfigure}
\usepackage{enumerate,enumitem}
\usepackage{amsmath,amsfonts,amsthm,amssymb}

\DeclareGraphicsExtensions{.pdf}

\graphicspath{{figs/}}

\theoremstyle{plain}
\newtheorem{theorem}{Theorem}
\newtheorem{lemma}[theorem]{Lemma}

\newtheorem*{conjecture}{Conjecture}

\theoremstyle{remark}

\let\leq\leqslant
\let\geq\geqslant
\let\setminus\smallsetminus

\makeatletter
\g@addto@macro\bfseries{\boldmath}
\makeatother

\newcommand{\fif}[1][1]{\ensuremath{1/5}}

\newcommand{\aB}{\ensuremath{{\rm avoidBob}}}
\newcommand{\dist}{\ensuremath{{\rm dist}}}

 \title{Playing weighted Tron on Trees}
 \author{Daniel Hoske}
 \author{Jonathan Rollin}
 \author{Torsten Ueckerdt}
 \author{Stefan Walzer}
 \address{Karlsruhe Institute of Technology}

 
 \date{\today}

\begin{document}

\maketitle

 \begin{abstract}
  We consider the weighted version of the \texttt{Tron} game on graphs where two players, Alice and Bob, each build their own path by claiming one vertex at a time, starting with Alice.
  The vertices carry non-negative weights that sum up to $1$ and either player tries to claim a path with larger total weight than the opponent.
  We show that if the graph is a tree then Alice can always ensure to get at most $\fif$ less than Bob, and that there exist trees where Bob can ensure to get at least $\fif$ more than Alice.
 \end{abstract}


\section{Introduction}

Mathematical games have been a frequent subject of study (see the $1700+$ articles recently collected by Fraenkel~\cite{F12}) for more than 60 years by now; not only because they are entertaining in nature, but also because of their relevance in practice, such as their close relation to diverse fields like biology, economics, and psychology.
Moreover, the analysis of mathematical games often reveals fundamental combinatorial structures and sometimes leads to the development of entire branches in mathematics and computer science.

In 1990 Bodlaender and Kloks~\cite{BK90} introduced the following impartial game with perfect information called \texttt{Tron}:
Given a graph, two players Alice and Bob choose distinct start vertices, first Alice then Bob.
Afterwards Alice and Bob take turns, each time claiming a new vertex adjacent to his/her previously claimed vertex and not claimed by anybody so far.
Thus either player builds a path with one end being its start vertex such that throughout the two paths are vertex disjoint.
A player loses when no vertex can extend the corresponding path.

\texttt{Tron} is inspired by the famous light cycle scene from the 1982 movie with the same name, which also gave rise to several computer games.
The combinatorial game \texttt{Tron} belongs to the class of \emph{subgraph building games}, where we essentially ask which of Alice and Bob can build a larger subgraph of a given graph according to certain building rules.
In the \emph{weighted} version, the game is played on \emph{instances} $(G,w)$, where $G$ is a graph and $w:V(G) \to \mathbb{R}_{\geq 0}$ are non-negative vertex weights, and one asks whether Alice or Bob can build the heavier subgraph, that is, one with larger total weight.
In this setting, it makes sense to allow either player to continue building his subgraph even if the other player already has no available move any more.
(Note that a player might win even by claiming significantly fewer vertices than the opponent.)
For convenience, instances are normalized (by scaling the weights) so that the total sum of weights in the graph equals~$1$, i.e., $\sum_{v \in V(G)}w(v) = 1$.
The \emph{outcome} of a game is the difference between the final weight of Bob's subgraph and Alice's subgraph.
So Bob tries to maximize the outcome, while Alice tries to minimize it.
The outcome of the game when both players play optimally is called the \emph{value} of that instance, denoted by $\Delta(G,w)$.
We are interested in the following extremal question for any fixed game:
\begin{center}
 \textit{For a given graph class $\mathcal{G}$, what is $\max_{G \in \mathcal{G}} \Delta(G,w)$ and $\min_{G \in \mathcal{G}} \Delta(G,w)$? I.e., what are the worst values for Alice and Bob, respectively?}
\end{center}

Concerning \texttt{Tron}, Miltzow~\cite{M12} constructed for every $n \in \mathbb{N}$ a graph $G_n$ on at least $n$ vertices in which Bob has a strategy to build a path on all but at most $8$ vertices, no matter how Alice plays.
Thus, $\Delta(G_n,w_n) \geq ((n-8) - 8)/n$, where $w_n$ assigns every vertex the uniform weight $1/|V(G_n)|$.
In particular, the value of such instances tends to $1$, so \texttt{Tron} can be arbitrarily bad for Alice on the set of all graphs, even in the unweighted case with uniform vertex weights.
On the other hand, if all but one vertex have a weight of zero, then such an instance is arbitrarily bad for Bob as its value is~$-1$.

In this paper we prove that the maximum value for \texttt{Tron} on the class of all trees equals $\fif$.
That is, playing \texttt{Tron} on a weighted tree, Alice can guarantee to claim a total weight of at most $\fif$ less than Bob.
And secondly, there is an instance in which Bob can guarantee to claim at least $\fif$ more than Alice.
Such an instance is given in Figure~\ref{fig:weighted-tree}.

\begin{theorem}\label{thm:main}
 Let $\mathcal{T}$ be the class of all trees. Then $\max_{T \in \mathcal{T}} \Delta(T,w) = \fif$.
\end{theorem}
    
\begin{figure}[htb]
 \centering
 \includegraphics{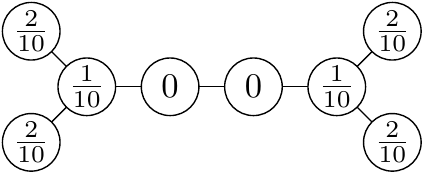}
 \caption{A weighted tree with value $\Delta = \fif$. Note that any starting position for Alice gives the same outcome, provided both play optimally.}
 \label{fig:weighted-tree}
\end{figure}

We shall prove Theorem~\ref{thm:main} in Section~\ref{sec:proof} below and conclude with two open problems in Section~\ref{sec:conclusions}.

\subsection*{Related work}

A central question in combinatorial game theory is to determine the computational complexity of deciding which player has a winning strategy, i.e., whether for a given instance $(G,w)$ we have $\Delta(G,w) \leq 0$ or $\Delta(G,w) \geq 0$.
For \texttt{Tron}, Bodlaender and Kloks~\cite{BK90} present a polynomial time algorithm when the graph is a tree, whereas Miltzow~\cite{M12} proves PSPACE-hardness for general graphs.

Related path building games include \texttt{Geography} where both players build the same path, always extending at the same end.
This variant, introduced by Fraenkel and Simonson in 1993~\cite{FS93}, resembles the children's game in which two players alternatingly have to find a country whose first letter matches the last letter of the previously country and that has not been taken yet.
For a summary of such path building games and their computational complexity we refer to the survey of Bodlaender~\cite{B93}.

In 2008 Peter Winkler stated the so-called \texttt{Pizza Game} in which the given graph is a vertex weighted cycle and both players build the same path which in each step may be extended on either end.
Again, if all the weight lies on one vertex, this is arbitrarily bad for Bob.
But Winkler asked whether Alice can always ensure $4/9$ of the total weight in the graph no matter how it is distributed, i.e., whether $\Delta(G,w) \geq 1/9$ whenever $G$ is a cycle.
In 2009 this has been confirmed independently by two sets of authors~\cite{CKMVS09,KMU11}.

For \texttt{Graph Sharing Games}, which are natural generalizations of the \texttt{Pizza Game}, it is known that there exist for every $k \geq 1$ a sequence of $k$-connected graphs whose values tend to $1$~\cite{CKMSV13}.
For trees, Micek and Walczak observe a parity phenomenon, namely that for \texttt{Graph Sharing Games} the minimum and maximum value for trees with even and odd number of vertices differ significantly~\cite{MW11,MW12}.
Let us remark that one variant of the \texttt{Graph Sharing Game} for trees has become well-known under the name \texttt{Gold Grabbing Game}~\cite{SS12}.

\section{Proof of Theorem~\ref{thm:main}}\label{sec:proof}

An \emph{instance} of \texttt{Tron} is a pair $(T,w)$ of a tree $T = (V,E)$ and non-negative real vertex weights $w:V \to [0,1]$ with $\sum_{v \in V} w(v) = 1$.
For a subset $U \subseteq V$ of vertices we denote $w(U) = \sum_{u \in U} w(u)$.
For a vertex $u \in V$, let $A_u$ and $B_u$ be the subsets of vertices that Alice and Bob take when Alice starts with $u$, and from then on Alice and Bob play optimally subject to minimizing $w(B_u)-w(A_u)$ and maximizing $w(B_u)-w(A_u)$, respectively.
Thus, $A_u \cap B_u = \emptyset$ and $w(A_u) + w(B_u) \leq 1$ for all $u\in V$.
The \emph{value of an instance $(T,w)$} is the difference $\Delta(T,w)$ in the final total weights of Bob and Alice, when both players play optimally.
In particular, $\Delta(T,w) = \min_{u \in V(G)} (w(B_u) - w(A_u))$.

For the remainder let $(T,w)$ be any fixed instance of \texttt{Tron}.
To prove Theorem~\ref{thm:main} we have to show that $\Delta = \Delta(T,w) \leq \fif$.
In Subsection~\ref{sub:partition} we partition $T$ into several (possibly empty) paths satisfying certain properties.
In Subsection~\ref{sub:avoid-Bob-strategies} we define the strategies for Alice, one of which will eventually ensure that she gets at most $\fif$ less than Bob.

\medskip

We start with the weighted version of a well-known fact about longest paths in trees, whose easy proof we omit here.
A \emph{heaviest path} is one with maximum total sum of weights.

\begin{lemma}\label{lem:heavy-path}
 If $v_0$ is any vertex in a tree $T$ with non-negative vertex weights and $P = (v_0,\ldots,v_k)$ is a heaviest among all paths starting in $v_0$, then there is a heaviest path of $T$ that starts in $v_k$.
\end{lemma}

\subsection{Partitioning the tree}\label{sub:partition}

For every vertex $u \in V$ let $B(u)$ be the vertex that Bob takes first when Alice starts with $u$ and Bob plays optimally.
I.e., the vertices in $B_u$ form a path starting with $B(u)$.
As $|E| < |V|$ there exists an edge $a_{\ell}a_r$ such that $B(a_{\ell})$ lies in the subtree rooted at $a_{\ell}$ containing $a_r$ and $B(a_r)$ lies in the subtree rooted at $a_r$ containing $a_{\ell}$.
The edge $a_{\ell}a_r$ splits the tree into two subtrees, which we call the \emph{left} and the \emph{right} \emph{side}, where $a_{\ell}$ and $a_r$ are part of the left and right side, respectively.

Let $P_{\ell} = (a_{\ell},\dots,b_{\ell})$ be a heaviest path in the left side starting at $a_{\ell}$.
By Lemma~\ref{lem:heavy-path} there is a heaviest path $Q_{\ell}$ in the left side starting in $b_{\ell}$.
It consists of an initial subpath~$Y_{\ell} \subseteq P_{\ell}$ and a second subpath $Z_{\ell}$ that is disjoint from $P_{\ell}$ (and possibly empty), where we identify paths in $T$ with their corresponding vertex sets.

We denote the endpoint of $Y_{\ell}$ different from $b_{\ell}$ by $d_{\ell}$ and the path $P_{\ell} \setminus Y_{\ell}$ by $X_{\ell}$.
The endpoint of $Q_{\ell}$ different from $b_{\ell}$ is called $c_{\ell}$.
See Figure~\ref{fig:tree-partition} for an illustration.
Lastly we define $R_{\ell}$ to be the set of those vertices of the left side not in $X_{\ell} \cup Y_{\ell} \cup Z_{\ell}$.
Symmetrically, we define $X_r, Y_r, Z_r, R_r, b_r, c_r, d_r$ for the right side.

\begin{figure}[htb]
 \centering
 \includegraphics{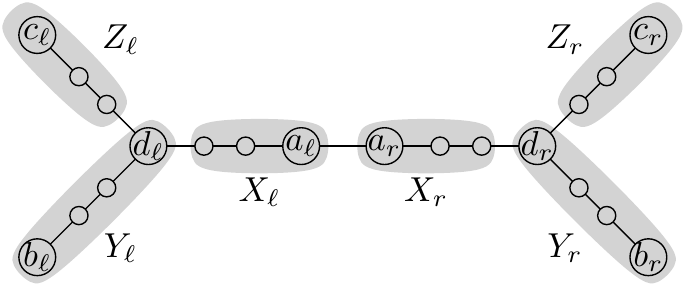}
 \caption{Left and right side are separated by the edge $a_{\ell}a_r$. The path $X_{\ell}\cup Y_{\ell}$ is a heaviest one in the left side starting with $a_{\ell}$. The path $Y_{\ell} \cup Z_{\ell}$ is a heaviest in the left side. And symmetrically we have $X_r,Y_r$ and $Z_r$ in the right side.}
 \label{fig:tree-partition}
\end{figure}

For convenience, we denote the weights of $X_{\ell}$, $Y_{\ell}$, $Z_{\ell}$, $R_{\ell}$, $X_r$, $Y_r$, $Z_r$, $R_r$ by $x_{\ell}$, $y_{\ell}$, $z_{\ell}$, $r_{\ell}$, $x_r$, $y_r$, $z_r$, $r_r$, respectively.
E.g., $x_r = w(X_r)$.
Then we have
\begin{equation}
 (x_{\ell} + y_{\ell} + z_{\ell} + r_{\ell}) + (x_r + y_r + z_r + r_r) = 1\label{eq:sumIsOne}
\end{equation}
as well as
\begin{equation}
 x_{\ell} \leq z_{\ell} \leq y_{\ell} \quad \text{and} \quad x_r \leq z_r \leq y_r.\label{eq:z-leq-y}
\end{equation}

\subsection{Avoid Bob Strategies}\label{sub:avoid-Bob-strategies}

Next we define the strategies for Alice that we use to prove Theorem~\ref{thm:main}.
For any two (not necessarily distinct) vertices $u,v$ the strategy \emph{avoid Bob after the path from $u$ to $v$}, denoted by $\aB(u,v)$, is defined as follows.
\begin{enumerate}[label = (\roman*)]
 \item Start at $u$ and walk towards $v$.
 \item When at $v$, proceed with a heaviest path from $v$ that avoids the subtrees of $T-v$ containing $u$ and $B(u)$.
\end{enumerate}
In case $u = v$ we also write $\aB(u)$ instead of $\aB(u,u)$.

Suppose Alice plays $\aB(a_{\ell})$.
Then she gets a total weight of $x_{\ell} + y_{\ell}$, since Bob answers in the right side.
In the worst case (for Alice) Bob gets a heaviest path in the right side, i.e., a total weight of at most $y_r + z_r$.
From $\Delta \leq w(B_{a_{\ell}}) - w(A_{a_{\ell}})$ and the symmetrical analysis of $\aB(a_r)$ we conclude that
\begin{align}
 \Delta &\leq (y_r + z_r) - (x_{\ell} + y_{\ell})\label{eq:aB-a1}\\
 \text{and } \quad \Delta &\leq (y_{\ell} + z_{\ell}) - (x_r + y_r).\label{eq:aB-a2}
\end{align}

In the remainder we shall consider certain avoid Bob strategies, derive inequalities like~\eqref{eq:aB-a1},~\eqref{eq:aB-a2} and combine these into an upper bound for $\Delta = \Delta(T,w)$, eventually proving Theorem~\ref{thm:main}, i.e., $\Delta \leq \fif$.
Note that when we derive from a set $S$ of strategies a set of inequalities which can then be combined to give $\Delta \leq \fif$, then this means that with at least one of the strategies in $S$ Alice's outcome is at most $\fif$.
So rather than defining a specific strategy for Alice ensuring $\Delta \leq \fif$ we give a set of avoid Bob strategies and prove that one of these ensures $\Delta \leq \fif$.
Similar approaches has been successfully used before~\cite{KMU11}.

As an example of this proof technique, let us sum $2$ times the inequalities~\eqref{eq:aB-a1} and~\eqref{eq:aB-a2}.
We obtain
\begin{equation*}
 4 \Delta \leq 2(z_r + z_{\ell} - x_r - x_{\ell}) \overset{x_{\ell},x_r \geq 0}{\leq} 2 z_r + 2 z_{\ell} \overset{\eqref{eq:z-leq-y}}{\leq} y_r + z_r + y_{\ell} + z_{\ell} \overset{\eqref{eq:sumIsOne}}{\leq} 1,
\end{equation*}
which implies $\Delta \leq 1/4$; a first upper bound for $\Delta$.
In other words, this weighted averaging argument shows that with one of $\aB(a_{\ell})$ and $\aB(a_r)$ Alice gets at most $1/4$ less than Bob.

We remark that this is tight if we consider $a_{\ell}$ and $a_r$ as possible starting positions for Alice only.
For example, playing the instance in Figure~\ref{fig:trickytree}, Bob can guarantee to get at least $1/4$ more than Alice, if she starts in $a_{\ell}$ or $a_r$.
(In general the edge $a_{\ell}a_r$ might not be unique, though in the example it is.)
    
\begin{figure}[htb]
 \centering
 \includegraphics{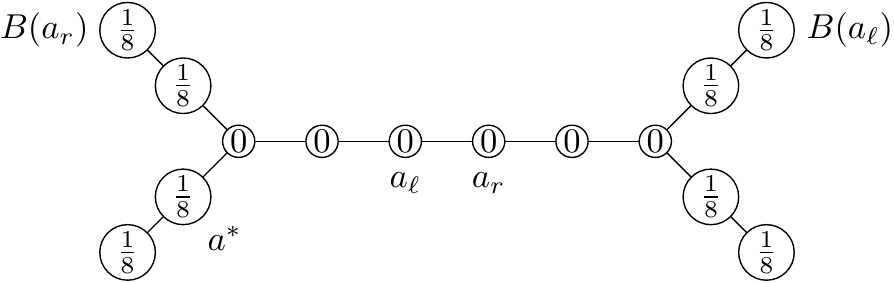}
 \caption{Both $a_{\ell}$ and $a_r$ are suboptimal starting positions for Alice. An optimal strategy starts for example at $a^*$.}
 \label{fig:trickytree}
\end{figure}

Of course, Alice can obey the avoid Bob strategy $\aB(u,v)$ only if she reaches the vertex $v$ before Bob.
We say that $\aB(u,v)$ is \emph{applicable} if $\dist(u,v) \leq \dist(B(u),v)$, where $\dist(u,v)$ denotes the (unweighted) distance between $u$ and $v$ in the tree.
We will ensure to use $\aB(u,v)$ only when it is applicable.
    
\begin{lemma}\label{lem:distances}
 If $\aB(a_r,d_{\ell})$ is applicable, then $\Delta \leq \fif$.
\end{lemma}
\begin{proof}
 We show that with one of $\aB(a_r,d_{\ell})$, $\aB(a_{\ell})$, $\aB(a_r)$ Alice gets at most $\fif$ less than Bob.
 
 When Alice plays $\aB(a_r,d_{\ell})$ and reaches $d_{\ell}$, $B(a_r)$ can not be in $Y_{\ell}$ and in $Z_{\ell}$ at once.
 So Alice gets at least the weight of the lighter one, which by~\eqref{eq:z-leq-y} is $z_{\ell}$.
 On the other hand, Bob gets at most $1 - (x_r + y_r + z_r + r_r + x_{\ell} + z_{\ell})$, since the right side is inaccessible to him and Alice takes at least $x_{\ell}$ and $z_{\ell}$ from the left side.
 We conclude that
 \begin{equation*}
  \Delta \leq 1 - (x_r + y_r + z_r + r_r + x_{\ell} + z_{\ell}) - (x_{\ell} + z_{\ell}) \leq 1 - y_r - z_r - 2z_{\ell}.
 \end{equation*}
 Adding twice the inequalities~\eqref{eq:aB-a1} and~\eqref{eq:aB-a2} (here we use the strategies $\aB(a_{\ell})$ and $\aB(a_r)$) we get
 \begin{align*}
  5\Delta \leq 1 - y_r - z_r - 2z_{\ell} + 2(z_r + z_{\ell}) = 1 - (y_r - z_r) \overset{\eqref{eq:z-leq-y}}{\leq} 1,
 \end{align*}
 which implies $\Delta \leq \fif$, as desired.
\end{proof}

By symmetry, Lemma~\ref{lem:distances} and its proof still hold if we switch all occurrences of $\ell$ and $r$ in the subscripts, i.e., if $\aB(a_{\ell},d_r)$ is applicable, then also $\Delta \leq \fif$.
For simplicity, we formulate and prove this and all subsequent lemmas only with respect to the left side of the tree, even though we also need their ``dual versions'' in the final proof of Theorem~\ref{thm:main}.

In some trees, like the one in Figure~\ref{fig:trickytree}, both strategies $\aB(a_{\ell},d_r)$ and $\aB(a_{\ell},d_r)$ are not applicable.
Moreover both $a_{\ell}$ and $a_r$ are suboptimal starting positions for Alice.
So we identify another starting position for Alice next.
Recall that $Q_{\ell} = Y_{\ell} \cup Z_{\ell}$ is a path from $b_{\ell}$ through $d_{\ell}$ to $c_{\ell}$ (see Figure~\ref{fig:tree-partition}).
We consider the real number \begin{equation}
 \alpha_{\ell} := \frac{1}{3}(r_{\ell} + 2y_{\ell} + z_{\ell} + x_{\ell} + x_r - z_r).\label{eq:alpha}
\end{equation}

In case $\alpha_{\ell} < z_{\ell} - x_{\ell}$, let $e_{\ell} \in Q_{\ell}$ be the vertex for which the paths from $e_{\ell}$ to $b_{\ell}$ and $c_{\ell}$ have weight at least $q_{\ell} - \alpha_{\ell}$ and $\alpha_{\ell}$, respectively, or vice versa.
Subject to that, $e_{\ell}$ shall be as close as possible to $d_{\ell}$.
We refer to Figure~\ref{fig:new-start-position} for an illustration.
Since $\alpha_{\ell} \leq z_{\ell} \leq y_{\ell}$, it is the path from $e_{\ell}$ to $b_{\ell}$ or $c_{\ell}$ that contains $d_{\ell}$ which has weight at least $q_{\ell} - \alpha_{\ell}$, while the other path \emph{not} containing $d_{\ell}$ has weight at least $\alpha_{\ell}$.
    
\begin{figure}[htb]
 \centering
 \includegraphics{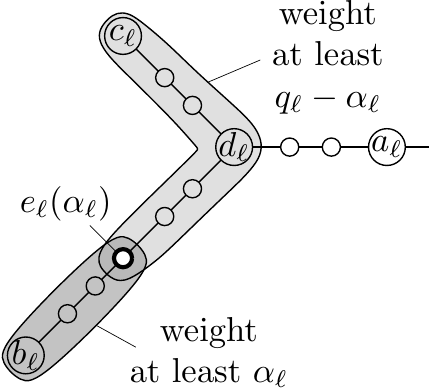}
 \caption{The path $Q_{\ell}$ connects $b_{\ell}$ and $c_{\ell}$. Consider the (essentially two) points such that going from them in one direction along $Q_{\ell}$ yields a weight of at least $\alpha_{\ell}$, while going in the other direction yields at least $q_{\ell} - \alpha_{\ell}$. Such a point closest to $d_{\ell}$ is $e_{\ell}$.}
 \label{fig:new-start-position}
\end{figure}

Note that $3\alpha_{\ell} \geq y_{\ell} + z_{\ell} - z_r \overset{\eqref{eq:z-leq-y}}{\geq} (y_{\ell} + z_{\ell}) - (y_r + x_r) \overset{\eqref{eq:aB-a2}}{\geq} \Delta$, which implies $\alpha_{\ell} \geq 0$ (unless $\Delta < 0$, in which case we are done).

\begin{lemma}\label{lem:distances-2}
 If $\alpha_{\ell} < z_{\ell} - x_{\ell}$ and $\dist(a_r,d_{\ell}) \leq \dist(e_{\ell},d_{\ell})$, then $\Delta \leq \fif$.
\end{lemma}
\begin{proof}
 By Lemma~\ref{lem:distances} we can assume
 \[
  \dist(B(a_r),d_{\ell}) < \dist(a_r,d_{\ell}) \leq \dist(e_{\ell},d_{\ell}).
 \]
 So Bob surely misses a weight of $\alpha_{\ell}$ from the path $Q_{\ell}$.
 Even more, when Alice plays $\aB(a_r)$ and Bob answers with $B(a_r)$, then Bob misses a weight of $x_{\ell}+\alpha_{\ell}$ from the left side.
 Indeed, if Bob's final path $B_{a_r}$ is disjoint from $X_{\ell}$, then $w(B_{a_r}) \leq y_{\ell} + z_{\ell} + r_{\ell} - \alpha_{\ell}$.
 Otherwise,
 \[
  w(B_{a_r}) \leq y_{\ell} + x_{\ell} + r_{\ell} - \alpha_{\ell} \overset{\eqref{eq:z-leq-y}}{\leq} y_{\ell} + z_{\ell} + r_{\ell} - \alpha_{\ell}.
 \]
 Thus we obtain a new inequality for the strategy $\aB(a_r)$:
 \begin{equation}
  \Delta \leq (y_{\ell} + z_{\ell} + r_{\ell} - \alpha_{\ell}) - (x_r + y_r).\label{eq:new-start-position-new}
 \end{equation}

 Basically, we have subtracted the term $\alpha_{\ell} - r_{\ell}$ from the corresponding inequality~\eqref{eq:aB-a2}.
 However, since $\alpha_{\ell} - r_{\ell}$ may be negative, we do not simply replace~\eqref{eq:aB-a2} with~\eqref{eq:new-start-position-new}.
 Instead we sum $9$ times~\eqref{eq:new-start-position-new} and $8$ times~\eqref{eq:aB-a2} and $13$ times~\eqref{eq:aB-a1} and obtain
 \begin{align*}
  30\Delta &\leq 9(y_{\ell} + z_{\ell} + r_{\ell} - \alpha_{\ell}) - 9y_r + 8(y_{\ell} + z_{\ell} - y_r) + 13(y_r + z_r - y_{\ell})\\
  &\overset{\eqref{eq:alpha}}{\leq} - 2y_{\ell} + 14z_{\ell} - 4y_r + 16z_r + 6r_{\ell}\\
  &\overset{\eqref{eq:z-leq-y}}{\leq} 6y_{\ell} + 6z_{\ell} + 6y_r + 6z_r + 6r_{\ell} \overset{\eqref{eq:sumIsOne}}{\leq} 6.
 \end{align*}
 
 Here we omitted all occurrences of $x_{\ell}$ and $x_r$ since these appear with negative signs only.
 Hence $\Delta \leq \fif$ as desired.
\end{proof}

Symmetrically to $\alpha_{\ell}$ and $e_{\ell}$ we can define $\alpha_r$ and (in case $\alpha_r < z_r - x_r$) $e_r$ and conclude the ``dual version'' of Lemma~\ref{lem:distances-2}, i.e., that if $\dist(a_{\ell},d_r) \leq \dist(e_r,d_r)$, then $\Delta \leq \fif$.

\begin{lemma}\label{lem:almost-last-bound}
 We have $\Delta \leq \fif$ or $\Delta \leq (y_r + z_r) - (y_{\ell} + z_{\ell}) + \alpha_{\ell}$.
\end{lemma}
\begin{proof}
 In case that $\alpha_{\ell} \geq z_{\ell}-x_{\ell}$ we immediately obtain
 \[
  \Delta \overset{\eqref{eq:aB-a1}}{\leq} (y_r + z_r) - (x_{\ell} + y_{\ell}) \overset{\alpha_{\ell} \geq z_{\ell}-x_{\ell}}{\leq} (y_r + z_r) - (y_{\ell} + z_{\ell}) + \alpha_{\ell}.
 \] 
 If $\alpha_{\ell} < z_{\ell}-x_{\ell}$ (so $e_{\ell}$ is well-defined) we consider strategies in which Alice starts with $e_{\ell}$ and distinguish the following cases for Bob's starting position $B(e_{\ell})$.
 \begin{description}
  \item[Case~1: $\dist(B(e_{\ell}),d_{\ell}) < \dist(e_{\ell},d_{\ell})$.]{\ \\}
   As $\dist(e_{\ell},d_{\ell}) > 0$, we have $e_{\ell} \neq d_{\ell}$.
   Consider the strategy~$\aB(e_{\ell})$.
   As $B(e_{\ell})$ is closer to $d_{\ell}$ than $e_{\ell}$, Bob's starting position lies in the subtree of $T-e_{\ell}$ containing $d_{\ell}$.
   By the definition of $e_{\ell}$ Alice gets at least a weight of $\alpha_{\ell}$ from $Q_{\ell}$.
   Moreover, $\dist(B(e_{\ell}),d_{\ell}) < \dist(e_{\ell},d_{\ell}) \leq \dist(a_r,d_{\ell})$ by Lemma~\ref{lem:distances-2}, and hence Bob starts his path $B_{e_{\ell}}$ in the left side.
      
   If (\textbf{Case~1.1}) $a_r \notin B_{e_{\ell}}$, i.e., Bob's path is contained in the left side, then (like in the previous proof) Bob misses a weight of $x_{\ell}+\alpha_{\ell}$ from the left side.
   Indeed, if $B_{a_r} \cap X_{\ell} = \emptyset$, then $w(B_{e_{\ell}}) \leq y_{\ell} + z_{\ell} + r_{\ell} - \alpha_{\ell}$.
   Otherwise, $w(B_{e_{\ell}}) \leq y_{\ell} + x_{\ell} + r_{\ell} \leq y_{\ell} + z_{\ell} + r_{\ell} - \alpha_{\ell}$, where the last inequality uses the assumption $\alpha_{\ell} \leq z_{\ell} - x_{\ell}$.
   
   Hence we get
   \begin{equation}
    \Delta \leq (y_{\ell} + z_{\ell} + r_{\ell} - \alpha_{\ell}) - \alpha_{\ell} = y_{\ell} + z_{\ell} + r_{\ell} - 2\alpha_{\ell}.\label{eq:case-1-1}
   \end{equation}
   
   Summing $9$ times~\eqref{eq:case-1-1} and $4$ times~\eqref{eq:aB-a1} and $2$ times~\eqref{eq:aB-a2} gives:
   \begin{align*}
    15\Delta &\leq 9(y_{\ell} + z_{\ell} + r_{\ell}) - 18\alpha_{\ell} + 2(z_r + z_{\ell}) + 2(y_{\ell} + z_{\ell}) - 2y_r \\
    &= 11y_{\ell} + 13z_{\ell} + 9r_{\ell} - 6(r_{\ell} + 2y_{\ell} + z_{\ell} + x_{\ell} + x_r - z_r) - 2y_r + 2z_r \\
    &\leq  y_{\ell} + 7z_{\ell} + 3r_{\ell} - 2y_r + 8z_r\\
    &\overset{\eqref{eq:z-leq-y}}{\leq} 3y_{\ell} + 3z_{\ell} + 3r_{\ell} + 3y_r + 3z_r \overset{\eqref{eq:sumIsOne}}{\leq} 3,
   \end{align*}
   which implies $\Delta \leq \fif$.
            
   If (\textbf{Case~1.2}) $a_r \in B_{e_{\ell}}$, i.e., Bob's path enters the right side, then $w(B_{e_{\ell}}) \leq y_{\ell} + x_{\ell} + r_{\ell} - \alpha_{\ell} + y_r + x_r$ and we conclude
   \[
    \Delta \leq (y_{\ell} + x_{\ell} + r_{\ell} - \alpha_{\ell} + y_r + x_r) - \alpha_{\ell} \overset{\eqref{eq:alpha}}{\leq} (y_r + z_r) - (z_{\ell} + y_{\ell}) + \alpha_{\ell}.
   \]
      
  \item[Case~2: $\dist(B(e_{\ell}),d_{\ell}) \geq \dist(e_{\ell},d_{\ell})$.]{\ \\}
   We consider the strategy~$\aB(e_{\ell},d_{\ell})$, because it is applicable.
   We further distinguish the possible locations of Bob's starting position $B(e_{\ell})$.
      
   If (\textbf{Case~2.1}) $B(e_{\ell})$ is \emph{not} in the subtree of $T - d_{\ell}$ containing $a_{r}$, then, after reaching $d_{\ell}$, Alice can continue her path into the right side.
   In total she gets at least $z_{\ell} - \alpha_{\ell} + x_{\ell} + x_r + y_r$, while Bob gets at most $y_{\ell} + r_{\ell}$.
   We obtain
   \begin{equation}
    \Delta \leq (y_{\ell} + r_{\ell}) - (z_{\ell} - \alpha_{\ell} + x_{\ell} + x_r + y_r) \overset{\alpha_{\ell} \leq z_{\ell} - x_{\ell}}{\leq} y_{\ell} + r_{\ell} - y_r.\label{eq:case-2-1}
   \end{equation}
   
   Summing $3$ times~\eqref{eq:case-2-1} and $7$ times~\eqref{eq:aB-a1} and $5$ times~\eqref{eq:aB-a2} yields
   \begin{align*}
    15\Delta  &\leq 3(y_{\ell} + r_{\ell} - y_r) + 5(z_r + z_{\ell}) + 2(y_r + z_r) - 2y_{\ell} \\
    &= y_{\ell} + 5z_{\ell} + 3r_{\ell} - y_r + 7z_r \\
    &\overset{\eqref{eq:z-leq-y}}{\leq} 3y_{\ell} + 3z_{\ell} + 3r_{\ell} + 3y_r + 3z_r \overset{\eqref{eq:sumIsOne}}{\leq} 3.
   \end{align*}
   Thus $\Delta \leq \fif$ and we are done.

   So we may assume that $B(e_{\ell})$ is in the subtree of $T - d_{\ell}$ containing $a_{r}$.
   Hence Alice can continue her path on the left, giving $w(A_{e_{\ell}}) \geq y_{\ell} + z_{\ell} - \alpha_{\ell}$.
   In this case Bob can not claim any vertex from $Q_{\ell}$.
   
   If (\textbf{Case~2.2}) $B_{e_{\ell}}$ contains vertices from the left side of $T$, then $w(B_{e_{\ell}}) \leq x_{\ell} + r_{\ell} + x_r + y_r$ and hence
   \begin{equation}
    \Delta \leq (x_{\ell} + r_{\ell} + x_r + y_r) - (y_{\ell} + z_{\ell} - \alpha_{\ell}) \overset{\alpha_{\ell} \leq z_{\ell} - x_{\ell}}{\leq} - y_{\ell} + r_{\ell} + x_r + y_r.\label{eq:case-2-2}
   \end{equation}
      
   Similar like above we obtain $\Delta \leq \fif$ by summing $3$ times~\eqref{eq:case-2-2}, $5$ times~\eqref{eq:aB-a1} and $7$ times~\eqref{eq:aB-a2} as follows
   \[
   15\Delta  \leq - y_{\ell} + 7z_{\ell} + 3r_{\ell} + y_r + 5z_r \overset{\eqref{eq:z-leq-y}}{\leq} 3y_{\ell} + 3z_{\ell} + 3r_{\ell} + 3y_r + 3z_r \overset{\eqref{eq:sumIsOne}}{\leq} 3.
   \]

   Finally, if (\textbf{Case~2.3}) $B_{e_{\ell}}$ has no vertex from the left side of $T$, then $w(B_{e_{\ell}}) \leq y_r + z_r$ (a heaviest path in the right side), and we obtain
   \begin{equation*}
    \Delta \leq (y_r + z_r) - (y_{\ell} + z_{\ell} - \alpha_{\ell}),
   \end{equation*}
   as desired.\qedhere
 \end{description}
\end{proof}

We are now ready to prove our main theorem, namely that $\Delta(T,w) \leq \fif$ for any instance $(T,w)$ of \texttt{Tron}.

\begin{proof}[Proof of Theorem~\ref{thm:main}]
 By symmetry we also have the ``dual versions'' of the preceding lemmas and hence also the ``dual version'' of Lemma~\ref{lem:almost-last-bound} where all occurrences of $\ell$ and $r$ in the subscripts are switched.
 So with Lemma~\ref{lem:almost-last-bound} we can conclude that $\Delta \leq \fif$ or
 \begin{align*}
  \Delta &\leq (y_r + z_r) - (y_{\ell} + z_{\ell}) + \alpha_{\ell}\\
  \text{and } \quad \Delta &\leq (y_{\ell} + z_{\ell}) - (y_r + z_r) + \alpha_r.
 \end{align*}
 Summing $3$ times these two inequalities yields
 \begin{equation}
  6\Delta \leq 3\alpha_{\ell} + 3\alpha_r = (r_{\ell} + 2y_{\ell} + z_{\ell} + x_{\ell} + x_r - z_r) + (r_r + 2y_r + z_r + x_r + x_{\ell} - z_{\ell}),\label{eq:last-inequality} 
 \end{equation}
 where we plugged in the definition of $\alpha_{\ell}$ and $\alpha_r$.
 Taking~\eqref{eq:last-inequality} and adding $2$ times inequalities~\eqref{eq:aB-a1} and~\eqref{eq:aB-a2} we finally obtain
 \begin{align*}
  10\Delta &\leq 3\alpha_{\ell} + 3\alpha_r + 2z_{\ell} + 2z_r \\
   &= r_{\ell} + 2y_{\ell} + 2z_{\ell} + 2x_{\ell} + r_r + 2y_r + 2z_r + 2x_r\overset{\eqref{eq:sumIsOne}}{\leq} 2,
 \end{align*}
 which implies $\Delta \leq \fif$, as desired.
\end{proof}

\section{Open Problems}\label{sec:conclusions}

In this paper we proved that playing \texttt{Tron} on weighted trees, Alice can guarantee to get at most $\fif$ less than Bob.
And this is tight, as verified by the instance in Figure~\ref{fig:weighted-tree}.
However, it remains open to determine the worst-case instance for \emph{unweighted} trees, i.e., where all vertex weights equal $1/n$, when $n$ is the number of vertices in the tree.
Playing such instances, Bob can always guarantee to get at most one vertex less than Alice by considering a longest path through Alice's starting position and playing next to it in the longer half.
The worst example (in terms of Alice) that we know of is shown in Figure~\ref{fig:unweighted-tree}, where Bob can guarantee to get $1/10$ of the total weight more than Alice.
Indeed, there are $40$ vertices and Bob can claim $4$ vertices more than Alice as follows:
Alice starts with some vertex $w$, say in the right side.
If $w \in \{u,v\}$, then Bob answers with $x$, while if $w \neq u,v$, then Bob answers with the vertex adjacent to $w$ in the subtree containing $u,v$.
Afterwards, Bob always continues in the direction containing a longest possible extension of his path.

\begin{figure}[htb]
 \centering
 \subfigure[\label{fig:unweighted-tree}]{
 \includegraphics{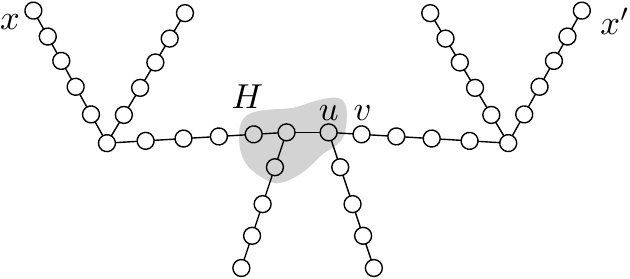}
 }
 \hspace{5em}
 \subfigure[\label{fig:weighted-cycle}]{
  \includegraphics{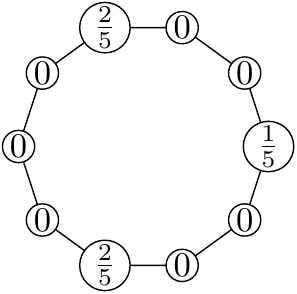}
 }
 \caption{\subref{fig:unweighted-tree} An instance of \texttt{Tron} with an unweighted tree and value $\Delta = 1/10$.
  \subref{fig:weighted-cycle} An instance of \texttt{Tron} with a weighted cycle and value $\Delta = \fif$.}
\end{figure}

On the other hand, Alice can start with $v$ and if Bob answers with one of the four vertices in $H$, then Alice can claim the path ending at $x'$. If he answers in the right side but not with $u$, she can claim the path ending at $x$.
In all other cases, it is enough for Alice to move towards Bob as long as possible, and continue with a longest path afterwards.
Using these strategies one can check that the tree in Figure~\ref{fig:unweighted-tree} has value $1/10$ and we conjecture that this is the worst case for unweighted trees.

\begin{conjecture}
 Any unweighted instance of \texttt{Tron} on trees has value at most $1/10$.
\end{conjecture}

It is also interesting to determine the maximum value of \texttt{Tron} instances on other classes of graphs, for example, on weighted cycles.
Oddly enough, the worst example of a weighted cycle that we know of (shown in Figure~\ref{fig:weighted-cycle}) has a value of $\fif$, just like the worst tree.

\begin{conjecture}
 Any instance of \texttt{Tron} on a cycle has value at most $\fif$.
\end{conjecture}

\bibliographystyle{elsarticle-num}
\bibliography{lit}

\end{document}